\documentclass[12pt]{amsart}
\usepackage[mathscr]{eucal}
\usepackage{amsmath,amsfonts}
\usepackage[all]{xy}
\usepackage{color}

\parskip=\smallskipamount

\newtheorem{theorem}{Theorem}[section]
\newtheorem{proposition}[theorem]{Proposition}

\theoremstyle{definition}

\newtheorem{remark}[theorem]{Remark}

\makeatletter
\@addtoreset{equation}{section}
\makeatother

\newcommand{\CC}{{\mathbb C}}
\newcommand{\NN}{{\mathbb N}}

\newcommand{\cB}{{\mathcal B}}

\newcommand{\cE}{{\mathcal E}}

\newcommand{\ra}{\rightarrow}

\newcommand{\tr}{\operatorname{tr}}
\let\phi=\varphi

\title{Interpolation for completely positive maps:
numerical solutions}

\author{C\u alin-Grigore Ambrozie}\thanks{The first named author's research 
was supported by the grants RVO:67985840 and 
CNCSIS UEFISCDI PN-II-ID-PCE-2011-3-0119}

\address{Institute of Mathematics - Romanian Academy, PO Box 1-764, RO 014700 
Bucharest, Romania \emph{and} 
Institute of Mathematics of the Czech Academy, Zitna 25, 11567 Prague 1, Czech 
Republic}
\email{calin.ambrozie@imar.ro}

\author{Aurelian Gheondea}\thanks{The second named author's research 
supported by a grant of the Romanian 
National Authority for Scientific Research, CNCSIS UEFISCDI, project number
PN-II-ID-PCE-2011-3-0119.}
\address{Department of Mathematics, Bilkent University, 06800 Bilkent, Ankara, 
Turkey, \emph{and} Institutul de Matematic\u a al Academiei Rom\^ane, C.P.\ 
1-764, 014700 Bucure\c sti, Rom\^ania} 
\email{aurelian@fen.bilkent.edu.tr \textrm{and} A.Gheondea@imar.ro} 



\begin{document}
\maketitle

\begin{abstract}
We present certain techniques to find completely positive maps between matrix 
algebras that take prescribed values on  given data.
To this aim we describe a semidefinite programming approach and another convex  
minimization method supported by a numerical example.
\end{abstract}

\section{Introduction}

The present paper refers to a certain interpolation problem for completely positive 
maps  that take prescribed values on given matrices, closely related to 
problems recently considered by C.-K.~Li and Y.-T. Poon in \cite{LiPoon}, Z.~Huang,  
C.-K.~Li, E.~Poon,  and N.-S.~Sze in \cite{HLPZ},  T.~Heinosaari, 
M.A.~Jivulescu, D.~Reeb,  and M.M.~Wolf   in \cite{hei}
as well as  G.M.~D'Ariano and P.~Lo Presti \cite{DArianoLoPresti}, 
D.S. Gon\c calves et al.~\cite{Goncalves}. 

Let $M_n$ denote the $C^*$-algebra 
of all $n\times n$ complex matrices. In particular, positive elements
(positive semidefinite matrices) in $M_n$ are defined. Recall that
a matrix $A\in M_n$ is \emph{positive semidefinite} if all its 
principal determinants are nonnegative. Let $M_{n}^+ \, \subset M_n$ denote the convex cone of all such matrices.
A linear map $\phi\colon M_n\ra M_k$ is \emph{positive} if $\varphi (M_{n}^+ )\subset M_{n}^+$, namely it maps positive semidefinite matrices into positive semidefinite ones.
We call $\varphi$  \emph{completely positive} if
the 
map $I_m\otimes \phi\colon M_m\otimes M_n\ra M_m\otimes M_k$ is 
positive for all $m\in \NN$. 
 
An equivalent notion is that of 
\emph{positive semidefinite} map, 
that is, for all $m\in\NN$, all $h_1,\ldots,h_m\in \CC^k$ and all 
$A_1,\ldots,A_m\in M_n$ we have
$\sum_{i,j=1}^m \langle \phi(A_j^*A_i)h_j,h_i\rangle  \geq 0$.
Let $\mathrm{CP}
(M_n,M_k)$ 
 denote the convex cone of all completely positive maps $\phi\colon M_n\ra M_k$. 
If $\phi:M_n\rightarrow M_k$ is completely 
positive then, cf.\ K.~Kraus \cite{Kra71} and M.D.~Choi \cite{Choi}, 
there are $n\times k$ matrices 
$V_1,V_2,\ldots,V_m$ with $m\leq nk$ such that
\begin{equation}\label{e:kraus}
\phi(A)=V_1^*AV_1+V_2^*AV_2+\cdots +V_m^*AV_m \mbox{ for all }
A\in M_n
\end{equation}
(and, of course, any map of the form (\ref{e:kraus}) is completely positive).
The representation \eqref{e:kraus} is called the \emph{Kraus representation} of 
$\phi$ and  $V_1,\dots,V_m$ are called the \emph{operation elements}. 
The representation \eqref{e:kraus} of a given completely positive 
map $\phi$ is  non-unique, with respect to both its operation 
elements and the number $m$ of these elements.
However the minimal number of the operation elements in the Kraus 
form of such a map $\varphi$ turns to be the rank of its Choi matrix (see subsection \ref{ss:tcm})  
- the statement is implicit in the original article of M.D.~Choi 
\cite{Choi}.
 The following problem has been
suggested by C.-K.~Li and Y.-T. Poon in \cite{LiPoon}, 
where a solution was given  in case when the  
families of matrices $(A_\nu )_\nu$, $(B_\nu )_\nu$ from below are commutative.

\noindent\textbf{Problem A.} \emph{
Given matrices $A_\nu \in M_n$ and $B_\nu \in M_k$ for $\nu =1,\ldots ,N$, 
find
$\phi \in\mathrm{CP}(M_n ,M_k)$ subject to the conditions}
\begin{equation}\label{e:princ}\varphi(A_\nu) =B_\nu,\mbox{ \it for all }\nu=1,\ldots,N.
\end{equation}
Other linear affine restrictions on $\varphi$ may be added as well, 
like trace preserving etc.
In \cite{prima} we dealt with various necessary and/or sufficient  conditions for the 
existence of such solutions $\phi$.
Most of the important theoretic results in this sense are related to  Arveson's 
Hahn-Banach type theorem \cite{Arveson} and
various techniques of operator spaces  \cite{Pau},
some of which being  simplified in the present particular context  by R.R.~Smith 
and J.D.~Ward \cite{SmithWard}. In this  paper we present some concrete techniques  to  compute  solutions numerically.
\vspace{2 mm}

Briefly speaking, the existence of solutions to Problem A (or related ones)  
always turns to be equivalent to the fact that
 certain affine subspaces of matrices contain at least 
one positive semidefinite matrix; also, this can be characterized  by the positivity 
of certain related linear functionals. 
In particular, our Theorem 2.4 and Theorem 2.5 in  
\cite{prima} show such  characterizations  
in terms of a certain density matrix $D_\varphi$ of $\varphi$, see Section 2.4 in 
\cite{prima}.
The density matrix  $D_\varphi$ and the Choi matrix
 $\Phi_\varphi$ are related by the equality 
 $D_\varphi =U^*\overline{\Phi}_\varphi U$ where the symbol
 $\overline{\mbox{ }^{\mbox{ }}}$ denotes the complex conjugation 
and $U$ is a  unitary operator  coming from the two canonical identifications 
of  $\mathbb{C}^n \otimes \mathbb{C}^k$ with $\mathbb{C}^{nk}$, see 
Proposition 2.8 in \cite{prima}; 
in this article we chose to use the Choi matrix $\Phi_\varphi$ instead. 
The first step in our approach is to firstly derive an equivalent formulation in 
terms of existence of certain positive 
semidefinite matrices subject to linear affine restrictions, like the matrix $X$  
($ = \Phi_\varphi$) in Problem B from Subsection~\ref{ss:esp}.

In Subsection~\ref{ss:smsp} we briefly describe a method for solving Problem B
by known techniques of semidefinite programming. 
Further, by using results from
\cite{ca}, we describe methods for solving Problem B by convex 
minimization techniques, see Theorem~\ref{minimizare}. A numerical example that
illustrates Theorem~\ref{minimizare} is performed in Subsection~\ref{ss:ne}. 
The approach we used in \cite{prima}, through the Smith-Ward linear functional, 
allows us to point out another numerical method of solving Problem B by means 
of minimization of linear functionals subject to semidefinite constraints, see 
Proposition~\ref{p:test}. Finally, in Subsection~\ref{ss:ccd} we show that, under
the commutation assumptions, the semidefinite problem that we obtain here turns 
into a linear programming problem, hence explaining the results in \cite{LiPoon}
from this perspective as well.

If a more restrictive case of 
Problem A is considered, for example, when, in addition to the 
requirement that the 
solutions $\phi$ should be completely positive maps, one imposes the condition of
trace preserving, that is, $\phi$ must be a quantum channel, we note that this
version of Problem A leads to the same type of Problem B since, 
the additional trace preserving constrained is just another linear constrained. This
shows that all the numerical techniques that we describe in this article 
can be successfully applied to interpolation of quantum 
channels that take prescribed values on given data, without any 
essential modification.

Let us  mention that the positive semidefinite approach to Problem A has been 
previously observed also in \cite{DArianoLoPresti}, \cite{Goncalves}, and
 \cite{hei}, in different formulation. With respect to these articles, our present topics, 
like subsections~\ref{ss:scmt}, \ref{ss:ne}, and Proposition~\ref{p:test} (b), 
are new.

\section{Main results}
\label{s:mr}

Consider then the interpolation problem (\ref{e:princ}) for the 
given matrices $A_\nu \in M_n$ and $B_\nu \in M_k$ where
$\nu =1,\ldots ,N$.
 Firstly, we will translate it below  in terms of Choi matrices.

\subsection{The Choi matrix.} \label{ss:tcm}
Let $\{e_i^{(n)}\}_{i=1}^n$ be the canonical basis of $\CC^n$ ($n\in\NN$).
As usual, the linear space $M_{n,k}$ of all $n\times k$ matrices is identified 
with the vector space $\cB(\CC^k,\CC^n)$  of all linear transformations 
$\CC^k\ra \CC^n$ ($n,k\in\NN$).  Let 
$\{E_{i,j}^{(n,k)}\mid i=1,\ldots,n,\ 
j=1,\ldots,k\}\subset M_{n,k}$ be the matrix units of size $n\times k$, namely $E_{i,j}^{(n,k)}$ is the $n\times k$
matrix with all entries $0$ except for the $(i,j)$-th entry which is $1$.  
If $n=k$, we note  $E^{(k)}_{i,j}=E^{(k,k)}_{i,j}$. 

Given any linear map $\phi\colon M_n\ra M_k$, define a $kn\times kn$ matrix 
$\Phi_\phi$ by
\begin{equation}\label{e:choi} \Phi_\phi =[\phi(E_{l,m}^{(n)})]_{l,m=1}^n.
\end{equation}
In what follows we describe the mapping
$\phi\mapsto \Phi_\phi$, that appears  in J.~de~Pillis \cite{Pillis}, 
A.~Jamio\l kowski \cite{Jamiolkowski}, R.D.~Hill \cite{Hill}, and M.D.~Choi 
\cite{Choi}.
Use the lexicographic reindexing of 
$\{E_{i,j}^{(n,k)}\mid i=1,\ldots,n,\ 
j=1,\ldots,k\}$, more precisely
\begin{equation}\label{e:Es} 
\bigl(
E^{(n,k)}_{1,1},\ldots,E^{(n,k)}_{1,k},E^{(n,k)}_{2,1},\ldots,E^{(n,k)}_{2,k},
\ldots, E^{(n,k)}_{n,1},\ldots,E^{(n,k)}_{n,k}\bigr) 
=\bigl(\mathcal{E}_1,\mathcal{E}_2,....,\mathcal{E}_{nk} \bigr)
\end{equation}
Another form of this reindexing is 
\begin{equation}\label{reindexing} \cE_r=E_{i,j}^{(n,k)}\mbox{ where }
r=(j-1)k+i,\mbox{ for all }i=1,\ldots,n,\ j=1,\ldots,k.
\end{equation}
The formula
\begin{equation}\label{e:prima} \phi_{(i-1)k+m,(j-1)k+l}
=\langle \phi(E_{l,m}^{(n)})e_{i}^{(j)},e_m^{(k)}\rangle,\quad i,j=1,\ldots,k,\
l,m=1,\ldots,n,
\end{equation} and its inverse
\begin{equation}\label{e:adoua} 
\phi(C)=\sum_{r,s}^{nk} \phi_{r,s} \cE_r^*C\cE_s,\quad C\in 
M_{n},
\end{equation} establish a linear  bijection 
\begin{equation}\label{e:defin}
\cB(M_n,M_k)\ni\phi\mapsto \Phi_\phi=[\phi_{r,s}]_{r,s=1}^{nk}\in M_{nk}
\end{equation}  
 that induces an affine, order preserving bijection
\begin{equation}\label{e:cpmp}
\mathrm{CP}(M_n,M_k)\ni\phi\mapsto\Phi_\phi\in M_{nk}^+.
\end{equation}
Given $\phi\in\cB(M_n,M_k)$ we call the matrix $\Phi =\Phi_\phi$ as in \eqref{e:choi} 
\emph{the Choi matrix} of $\phi$.

\subsection{Equivalent setting of the problem.}\label{ss:esp}
 Following 
the notation in (\ref{e:choi}) -- (\ref{e:cpmp}), the Choi matrix $\Phi =\Phi_\varphi$ of any solution $\varphi :M_n \to M_k$ to (\ref{e:princ})  is given by
$\Phi =[\varphi_{rs}]_{r,s}$ where the indices $r,\, s$  are couples $r\equiv (i,m)$, $s\equiv (j,l)$ for $i,j=1,\ldots ,n$,  $l,m=1,\ldots ,k$ and $$\varphi_{rs}
=\langle \varphi (E_{ij}^{(n)})e_{l}^{(k)},e_{m}^{(k)}\rangle .$$ 
Since $r,\, s$ run  the cartesian product $\{ 1,\ldots ,n\}\times \{ 1,\ldots , k\}$
 consisting of $nk$ elements that we  order lexicographically, we can simply write $\Phi \in M_{nk}$ and   $r,s=1,\ldots ,nk$. 
Set
$$
A_\nu =[a_{\nu ,i,j}]_{i,j=1}^n =\sum_{i,j=1}^n a_{\nu ,i,j} E^{(n)}_{ij}
 $$
and 
$$
B_\nu =[b_{\nu ,m,l}]_{m,l=1}^ k =\sum_{m,l=1}^k b_{\nu ,m,l} E^{(k)}_{ml}.
 $$
Equate the $(m,l)$ entries in the matrix equality $\varphi (A_\nu) =B_\nu$ to get $$\langle \varphi (A_\nu )e_{l}^{(k)},e_{m}^{(k)}\rangle =b_{\nu, m, l}, $$ that is,
$$
\langle \varphi (\sum_{i,j=1}^n a_{\nu ,i,j}E^{(n)}_{ij} )e_{l}^{(k)},e_{m}^{(k)}\rangle =b_{\nu, m,l}
$$
and so 
\begin{equation}\label{simetric}
\sum_{i,j=1}^n a_{\nu ,i,j}\varphi_{(i,m)(j,l)}=b_{\nu ,m,l}.
\end{equation}
Write the equality from above using Kronecker's symbol $\delta_{pq}$ ($=1$ if $p=q$ and $0$ if $p\not =q$) in the form
$$
\sum_{(j,l'),(i,m')}a_{\nu ,i,j}\delta_{l'l}\delta_{m'm}\varphi_{(i,m')(j,l')}=b_{\nu ,m,l}
$$
where  $(j,l')$ and $ (i,m')$ run   $\{ 1,\ldots ,n\} \times \{ 1,\ldots ,k \}$,
then set
\begin{equation}\label{aiota}
c(\nu ,m,l)_{(j,l')(i,m')}:=a_{\nu ,i,j}\delta_{l'l}\delta_{m'm}=(A_{\nu}^\tau )_{j,i}(E_{lm}^{(k)})_{l',m'}=(A_{\nu}^\tau \otimes E_{lm}^{(k)})_{(j,l'),(i,m')}
\end{equation}
where $A\mapsto A^\tau$ denotes the transposition and define 
\begin{equation}\label{aiota2}
C(\nu ,m,l)=\big[ c(\nu ,m,l)_{(j,l')(i,m')}\big]_{(j,l')(i,m')}=A_{\nu}^{\tau}\otimes E^{(k)}_{lm}
\end{equation}
that can be  represented as an $nk\times nk$ matrix  $C(\nu ,m,l)\in M_{nk}$
 \begin{equation}\label{dtp}C(\nu ,m,l)\equiv A_{\nu}^\tau \otimes   E_{lm}^{(k)}\equiv \big[     a_{\nu ,  j,i  }E_{lm}^{(k)}  \big]_{i,j=1}^n  \end{equation}
  via the linear, isometric, order-preserving isomorphisms $$ M_{nk}\equiv M_n \otimes M_k \equiv M_n (M_k ).$$
 We obtain, using (\ref{simetric}) -- (\ref{aiota2}), the equality
$$
\sum_{(j,l'),(i,m')}c(\nu ,m,l)_{(j,l')(i,m')}\varphi_{(i,m')(j,l')}=b_{\nu ,m,l},
$$
namley
$$\tr\, (C(\nu ,m,l)\Phi )=b_{\nu ,m,l},
$$
that by  (\ref{dtp}) we can  write as well 
\begin{equation}\label{tens}
\tr\, [(A_{\nu}^\tau \otimes   E_{lm}^{(k)})\Phi ]=b_{\nu , m,l},
\end{equation}
which  actually is a particular application of the next formula, easily checked following the lines from above, 
letting $A=[a_{ij}]_{i,j=1}^n =\sum_{i,j=1}^n a_{ij}E_{ij}^{(n)}$ etc:
\begin{equation}\label{formm}
\varphi (A)=\big[ \tr\, [(A^\tau \otimes   E_{lm}^{(k)}){\Phi} ]\big]_{m,l=1}^k   =
\big[ \tr\, [(A \otimes   E_{lm}^{(k)}){\Phi^\tau} ]\big]_{m,l=1}^k 
 \, \mbox{ }\mbox{ }\mbox{ }(A\in M_n ).
\end{equation}
 Note that  we have  as well the formula $\varphi (A)=   [\tr\, [(E_{lm}^{(k)}\otimes A)D_{\varphi}^* ]_{m,l=1}^k$ where $D_\varphi$ denotes the density matrix \cite{prima},  for which we also omit the details.
Conditions (\ref{e:prima}) on $\varphi$ are thus equivalent to the equations (\ref{tens}) from 
above concerning $\Phi$, via the formulas (\ref{aiota}),  (\ref{aiota2}) and (\ref{e:prima}), (\ref{e:defin}). 
Denote by $\iota =(\nu ,m,l)$ the generic triple consisting of arbitrary $\nu =1,\ldots ,N$ and $m,l=1,\ldots ,k$.  Thus $\iota$ runs a set of $q:=Nk^2$ elements. We may   write
$\iota =1,\ldots ,q$. Set also $p=nk$. Write
$C(\iota )=C(\nu ,m,l)$ ($\in M_{p}$)
and  $b_\iota = b_{\nu, m,l}$. Via (\ref{e:cpmp}),
 Problem A takes then the form 
from below.
\vspace{3 mm}

\noindent {\bf Problem B} {\it Given  $C(\iota )\in M_p$ and numbers $b_\iota$ ($1\!\leq\!\iota \!\leq\! q$), find  $X \in M_p$, $X \geq 0$, such that
\begin{equation}\label{transl}
\tr\, (C(\iota )X)=b_\iota
\mbox{ }\mbox{  } \mbox{for all }\iota=1,\ldots,q .
\end{equation}
}
\vspace{3 mm}

\noindent Thus, the solvability of Problem  A leads to the rather known topic of finding  
positive semidefinite matrices subject to  linear affine conditions and, in particular, 
establish whether such matrices do exist.
These questions often occur and are dealt with by reliable numerical methods in the  
semidefinite programming, a few elements of which we sketch in what follows.

In addition, a more restrictive case of Problem A is when, in addition to the 
requirement that the 
solutions $\phi$ should be completely positive maps, one imposes the condition of
trace preserving, that is, $\phi$ must be a quantum channel. However, this
version of Problem A leads to the same type of Problem B since, 
the additional trace preserving constrained is just another linear constrained. This 
shows that Problem B can be successfully applied to interpolation of quantum 
channels that take prescribed values on  given data, as well.

\subsection{Solutions by means of semidefinite programming.}\label{ss:smsp}
Firstly, using
 $\tr\, (c^* )=\overline{\tr\, (c)}$, $\tr\, (cd)=\tr\, (dc)$ and writing equation (\ref{transl}) in terms of $C(\iota )+C(\iota )^*$  and
 $i(C(\iota )-C(\iota )^* )$  we can 
 asume  all  matrices $C(\iota )$ to be selfadjoint. 
We can  suppose, without loss of generality, 
that they are linearly independent over $\mathbb{R}$. 
 {\it Semidefinite programming} is concerned with 
minimization of linear functionals  subject to the constraint that an 
affine combination of selfadjoint matrices is positive semidefinite: see
in this sense \cite{Barvinok},  \cite{BoydGhaoui}, \cite{LundquistJohnson}, 
\cite{NesterovNemirovsky}, \cite{VanderbergheBoyd}, also 
\cite{BdGhFrBl}, \cite{GolubVanLoan}.  
Roughly speaking, one sets $$a(x)=\sum_\iota x_\iota C(\iota )+a_0$$ for the given 
$C(\iota )$ and a  selfadjoint matrix $a_0$ (that can be suitably chosen, here). Define then
$$
p^* =\inf_x \{ \sum_\iota b_\iota x_\iota  \, :\, a(x)\geq 0\}
$$
and
$$
q^* =\inf_X \{ -\tr\, (a_0 X): X\geq 0, \tr\, (C(\iota )X)=b_\iota\, \forall \iota \} .
$$
A problem dual to (\ref{transl}) occurs now with respect to $p^*$, namely to establish if there exist positive definite matrices of the form $a(x)$. Standard algorithms exist to this aim, based on maximizing the minimal eigenvalue of $a(x)$ in the variables $x=(x_\iota )_\iota$, or on  interior point methods using
 barrier functions \cite{NesterovNemirovsky}. 
In the case when either (\ref{transl}) has solutions $X>0$, or the dual problem has solutions $x$ with $a(x)>0$,
we have $$p^* =q^* ,$$ see for instance \cite{NesterovNemirovsky}, \cite{VanderbergheBoyd}. If both conditions hold, the optimal sets for $p^*$ and $q^*$ are nonempty. In this case for every $\lambda \in (p^* ,\overline{p})$ where $\overline{p}=\sup_x \{ \sum_\iota b_\iota x_\iota :a(x)>0\}$ there is a unique vector $x^* =(x_{\iota}^* )_\iota $, the {\it analytic center} of this linear matrix inequality, such that $$a(x^* )>0, \mbox{ }\mbox{  }\sum_\iota b_\iota x^{*}_\iota =\lambda$$ and $x^*$ minimizes the {\it logarithmic 
 barrier function} $$\ln \det a(x)^{-1}$$ over all $x$ with $\sum_\iota b_\iota x_\iota$ and $a(x)>0$. It follows by the Lagrange multipliers method that  $$\tr\, (C(\iota )a(x^* )^{-1})=\lambda b_\iota \, \mbox{ }\forall \,
\iota ,$$ which gives a solution  $X=X_*$ of
(\ref{transl}), namely 
\begin{equation}\label{sdpt}X_* =\lambda^{-1}a(x^* )^{-1}.\end{equation}
These techniques  provide then a  method to find solutions of the form (\ref{sdpt}) to Problem B.

\subsection{Solutions via a convex minimization technique.}\label{ss:scmt}
We present another way to obtain   particular solutions  to Problems A, B, based on the minimization 
of a certain convex function, see  \cite{ca}. % for  details.
% $V$, see 
%Theorem \ref{minimizare}. 
Suppose that $C(\iota )$ are selfadjoint and linearly independent.
Define   the function $V$ of $q$ real variables $x=(x_\iota )_{\iota =1}^q$ by
\begin{equation}\label{vv}
V(x)=\tr\, \big( e^{\sum_{\iota =1}^q x_\iota C(\iota )} \big)
-\sum_\iota x_\iota b_\iota .\end{equation}% where $e^A$ stands for the exponential of a matrix $A$ ($>0$ whenever all C(\iota )=C(\iota )^*$ as assumed)
 Then $V$ is smooth, strictly convex and has  strictly positive definite Hessian \cite{ca}. Hence whenever it has some critical point this is unique,
 and necessarily a point of minimum. Generally we may have also an unattained infimum $\inf V> -\infty$, 
or $\inf V=-\infty $.
 The following characterization of the existence of the solutions $X>0$ to 
 Problem B holds.

\begin{theorem}\label{minimizare}{\rm \cite{ca}}.
The system of equations (\ref{transl}) admits solutions $X>0$ if and only if the function $V$ defined by (\ref{vv}) has a critical point (of minimum), that is, if and only if $\lim_{\| x\| \to \infty }V(x)=+\infty$.
 In this case, (\ref{transl}) has also the (positive) particular solution
\begin{equation}\label{exp}
X_0 =e^{\sum_\iota x^{0}_\iota C(\iota )}
\end{equation}
where $x^0 =(x_{\iota}^0 )_\iota$ is the  critical point of $V$.
\end{theorem}

\begin{remark}\label{grad}
{\rm The function $V$ given by (\ref{vv})  fulfills the conditions of application of the
  {\it method of the conjugate gradients} \cite{HrLm}. This yields,
 whenever problem (\ref{transl}) has solutions $X>0$, a minimizing sequence
of vectors $x=(x_\iota )_\iota$ that is convergent to the critical point $x^0$ of $V$
 and so provides  approximations $\widetilde{X}_0 =e^{\sum_{\iota}x_\iota C(\iota )}\, \approx X_0$ of the solution  
(\ref{exp}), see  [Example 12, Remarks 11, \cite{ca}]. 
Note that the gradient $\nabla V =(\partial V/\partial x_\iota )_{\iota =1}^q$ of $V$ is easily computed to this aim
by $$\frac{\partial V}{\partial x_\kappa }(x)=  \tr\, (C(\kappa ) e^{\sum_\iota x_\iota C(\iota )} ) -b_\kappa ,$$ see \cite{ca}. 
We remind also the existence of various versions of  Newton's method that can be used as well to approximate the critical point. Certain tests exist \cite{ca} also to check if there are no solutions $X\geq 0$ at all.} 
\end{remark}

\subsection{A Numerical Example.}\label{ss:ne}
We show  how
Theorem \ref{minimizare}  applies to Problems A, B. 
Suppose one wishes to find   $\varphi :M_2 \to M_2$ completely positive
such that $\varphi (A_\nu) =B_\nu$ ($\nu =1,2$) for
$
 A_1 =\left[ \begin{array}{cc}2 & 1 \\ 1& 0\end{array}\right] , A_2 =\left[ \begin{array}{cc}1 & 1\\ 1& 2\end{array}\right] $ and  $
B_1 =\left[ \begin{array}{cc}4 & 0 \\ 0 & 0\end{array}\right] , B_2 =\left[ \begin{array}{cc}3.5 & 1.5 \\ 1.5 & 2.5 \end{array}\right] 
$.
 Use to this aim the minimization method indicated by Remark \ref{grad}.   
Formulas (\ref{aiota}), (\ref{aiota2}) and (\ref{dtp}) provide the matrices $C({\iota})$ 
for $\iota =(\nu ,m,l)$ where $\nu ,m,l=1,2$.
Due to the symmetry  equation (\ref{simetric}) (or, equivalently, (\ref{transl})) is 
equivalent to $\sum_{j,i=1}^n \overline{a}_{\nu ,ji}\overline{\varphi}_{(j,l)(i,m)}=
\overline{b}_{\nu lm}$ (or $\tr\, (C(\iota )^* \, \Phi )=\overline{b}_\iota$), and so it is 
enough to consider (\ref{simetric}) for those couples $(m,l)$ with $m\leq l$. That is, 
for each $\nu =1,2$ we have 3 equations corresponding to $(m,l)=(1,1),(1,2),(2,2)$.
 The set
 $\{ 1,2\} \times \{ 1,2\}$ of indices $r,\, s$ like $(j,l), (j,l'),(i,m), (i,m') $  with $1\leq i,j
 \leq n\, $($=2$) and $1\leq m,m',l,l'\leq k\, $($=2$) from below is   ordered 
 lexicographically as $\{ (1,1),(1,2),(2,1),(2,2)\} \equiv \{ {\bf 1},{\bf 2},{\bf 3},{\bf 4}\} $. 
We represent the positive matrix $X=[y_{\alpha \beta }]_{\alpha , \beta \in \{ {\bf 1},{\bf 
2},{\bf 3},{\bf 4}\}}\equiv \Phi =[\varphi_{rs}]_{r,s}$ that we seek for as
$$X=   \left[ \! \begin{array}{cccc}
y_{\bf 11} & y_{\bf 12} & y_{\bf 13} & y_{\bf 14} \\
y_{\bf 21} & y_{\bf 22} & y_{\bf 23} & y_{\bf 24} \\
y_{\bf 31} & y_{\bf 32} & y_{\bf 33} & y_{\bf 34} \\
y_{\bf 41} & y_{\bf 42} & y_{\bf 43} & y_{\bf 44} 
\end{array}\!\right] \equiv \Phi =
 \left[ \! \begin{array}{cccc}
\!\varphi_{( 1, 1)(1 ,1 )} & \!\varphi_{(1 ,1 )(1 ,2 )} & \!\varphi_{(1 ,1 )(2 ,1 )} & \!\varphi_{(1 ,1 )(2 ,2 )} \\
\!\varphi_{( 1, 2)(1 ,1 )} & \!\varphi_{(1 ,2 )(1 ,2 )} & \!\varphi_{(1 ,2 )(2 ,1 )} & \!\varphi_{(1 ,2 )(2 ,2 )} \\
\!\varphi_{( 2, 1)(1 ,1 )} & \!\varphi_{(2 ,1 )(1 ,2 )} & \!\varphi_{(2 ,1 )(2 ,1 )} & \!\varphi_{(2 ,1 )(2 ,2 )} \\
\!\varphi_{( 2, 2)(1 ,1 )} & \!\varphi_{(2 ,2 )(1 ,2 )} & \!\varphi_{(2 ,2 )(2 ,1 )} & \!\varphi_{(2 ,2 )(2 ,2 )} 
\end{array}\!\right]
$$
and the given matrices $C({\nu ,m,l})$  as follows: $E_{11}=\left[ \begin{array}{cc}
1&0\\0&0\end{array}\right]$, $E_{21}=\left[ \begin{array}{cc}0&0\\1&0\end{array}
\right]$ etc and 
$$
C({1,1,1})\! =\! \left[ \begin{array}{cccc} \! c({1,1,1})_{( 1,1 )(1 ,1 )} & \! c({1,1,1})_{(1 ,1 )(1 ,2 )} & \! c({1,1,1})_{(1 ,1 )(2 ,1 )} & \! c({1,1,1})_{(1 ,1 )(2 ,2 )} \\
 \! c({1,1,1})_{(1 ,2 )(1 ,1 )} & \! c({1,1,1})_{( 1,2 )(1 ,2 )} & \! c({1,1,1})_{(1 ,2 )(2 ,1 )} & \! c({1,1,1})_{(1 ,2 )(2,2 )} \\
 \! c({1,1,1})_{(2 ,1 )(1 ,1 )} & \! c({1,1,1})_{( 2, 1)(1 ,2 )} & \! c({1,1,1})_{(2 ,1 )(2 ,1 )} & \! c({1,1,1})_{(2 ,1 )(2,2 )} \\
 \! c({1,1,1})_{(2 , 2)(1 ,1 )} &\! c({1,1,1})_{(2 ,2 )(1 ,2 )} &\! c({1,1,1})_{(2 ,2 )(2 ,1 )} &\! c({1,1,1})_{(2 ,2 )(2 ,2 )} 
\end{array}\! \right] =$$ 
$$
A_{1}^\tau \otimes E_{11}=\left[ \begin{array}{cc} a^{1}_{11}\left[ \begin{array}{cc}1&0\\ 0& 0\end{array}\right] &  a^{1}_{21}\left[ \begin{array}{cc}1&0\\ 0& 0\end{array}\right] \\
\mbox{ } & \mbox{ }\\
 a^{1}_{12}\left[ \begin{array}{cc}1&0\\ 0& 0\end{array}\right] &  a^{1}_{22}\left[ \begin{array}{cc}1&0\\ 0& 0\end{array}\right] \end{array}\right] =
\left[ \begin{array}{cccc} a_{11}^1 & 0& a_{21}^1 & 0\\
0&0&0&0\\
a^{1}_{12}&0&a_{22}^1 &0\\
0&0&0&0\end{array}\right] =
\left[ \begin{array}{cccc} 2 & 0& 1 & 0\\
0&0&0&0\\
1&0&0 &0\\
0&0&0&0\end{array}\right]
,$$ 
$
C({1,1,2})=A_{1}^\tau \otimes E_{21} =\left[\begin{array}{cccc}
0 & 0 & 0 & 0\\
2 & 0 & 1 & 0\\
0 & 0 & 0 & 0\\
1 & 0 & 0 & 0
\end{array}  \right] 
$ etc. 
 A numerical minimization of $V$ as in Remark \ref{grad} gave us the matrix $\widetilde{X}_0 \approx X_0$ from below
$$
\widetilde{X}_0  =\left[
\begin{array}{rrrr}
1.549937761 &  -0.1694804138 &  0.4499571618 &  0.4047411695 \\ 
-0.1694804138 &  0.1534277390 &  -0.06572393508 &  -0.1533566973\\
0.4499571618 &  -0.06572393508 &  0.5249880063 &  0.6652436210\\ 
0.4047411695 &  -0.1533566973 &  0.6652436210 &  1.326699194
\end{array}\right]
$$
which approximately satisfies the equations (\ref{transl}). 
Let $\varphi$ be the map whose Choi matrix $\Phi =\Phi_\varphi =[\varphi_{r,s}]_{r,s}$ is $\widetilde{X}_0 \equiv \Phi$. We got then
an approximate solution to our present particular case of problem (\ref{e:princ}), namely $
\varphi (A)=\,\left [\, \sum_{i,j=1}^2 \Phi_{(i,m)(j,l)}a_{ij}\, \right]_{m,l=1}^2 =\left[ \tr [(A^\tau \otimes E^{(k)}_{lm})\widetilde{X}_0 \right]_{m,l=1}^k  
$
for every $A=[a_{ij}]_{i,j=1}^2 \in M_2 $ , see formula (\ref{formm}). 
For instance, we have   \begin{multline}\varphi (A_1 )_{11}=\!  \tr\, (C(1,1,1)\widetilde{X}_{0})=2(\widetilde{X}_0 )_{11}+(\widetilde{X}_0 )_{13}+(\widetilde{X}_0 )_{31} =3.99978984  
\approx 4 =b_{(1,1,1)}=(B_{1})_{11} \mbox{,}\nonumber 
\\
\varphi (A_1 )_{12}=\tr \left[ (A_{1}^\tau \otimes E_{21}^{(2)})\widetilde{X}_0 \right] =\tr (C(1,1,2)\widetilde{X}_0 )=0.0000564069 \approx 0=b_{(1,1,2)}=(B_1 )_{12}
\mbox{ etc}\nonumber\end{multline}

Definitely, problems of more sizeable amount can be solved as well by using such  semidefinite programming (or related) 
methods \cite{NesterovNemirovsky}, \cite{VanderbergheBoyd}, see also \cite{borch} ,  \cite{KocvaraStingl}, allowing us to consider  larger $n,\, k,\, N$.

\begin{remark}{\rm In order to obtain an exact solution $X$, we can
project  $\widetilde{X}_0 $  onto the affine subspace defined by (\ref{transl}) by  a linear affine projection map $p$, then let
$$
X:=p\widetilde{X}_0   
$$ and use $\Phi :=X$ instead of $\widetilde{X}_0 $. Indeed, since $\widetilde{X}_0 \approx X_0$ then $X=p\widetilde{X}_0 \approx pX_0 =X_0 $ and so $X\approx  X_0 \,  >0$ which implies   $X>0$ if a sufficiently good approximation $\widetilde{X}_0 \approx X_0$ 
 was performed.}
\end{remark}

\begin{remark} {\rm
An interesting question is to reduce the number of operation elements in the representation (\ref{e:kraus}) of  $\varphi$, whenever possible.
This is equivalent to the minimization of the rank of  $X$. The case of one term for instance would correspond to  solutions $X\geq 0$  of rank one, namely to the existence of vectors $v\in \mathbb{C}^{nk}$ such that $\langle C(\iota ) v,v\rangle =b_\iota$ for all $\iota$. 
A first easy step to rank reduction  is to find the joint support $P$ of the matrices $C(\iota )$ ($: =C({\iota})^*$) and consider only solutions $X$ such that $X=PXP$, as follows.
 Set $K=\{ h\in \mathbb{C}^{nk}: C(\iota )h=0\, \forall \, \iota \}$. Let $P$ be the orthogonal projection onto $K^\bot $. Then $\tr\, (C(\iota )PXP)=
b_\iota$ for all $\iota$. Indeed, $C(\iota )=C(\iota )P$ and so $C(\iota )=C({\iota})^* =PC(\iota )=PC(\iota )$P, whence $\tr\, (C(\iota )X)=\tr\, (PC(\iota )P X)=\tr\, (C(\iota )PXP)$. Generally the question to verify if there exist solutions $X\geq 0$ of lower rank and find them is difficult.
 For certain possibilities of reducing  the rank of $X$ see, for instance, Section II.13
 in \cite{Barvinok}, or \cite{SadatiYousefi}.}
\end{remark}

\subsection{Characterization in terms of linear functionals.}\label{ss:ctlf}
By  Theorem 2.5 from \cite{prima} (see also Theorem 6.1 in \cite{Pau}, or \cite{SmithWard}), the 
solvability of (\ref{transl}) can be described  in terms of the linear functional

$$\sum_\iota x_\iota C(\iota ) \mapsto b_\iota x_\iota .$$ 
We recall  this result in the  form from below,   completed with a version (b) 
concerning the existence of strictly positive solutions; 
for the sake of completeness we sketch also  the proof.

\begin{proposition}\label{p:test}
Suppose that $C(\iota ) \in M_p$ ($\iota =1,\ldots ,q$) are selfadjoint,  linearly 
independent and their linear span contains strictly positive matrices. Then:

\emph{(a)} There exist solutions $X\geq 0$ of the system of equations (\ref{transl}) if
and only if

\noindent $\sum_\iota b_\iota x_\iota  \geq 0$ for all $(x_\iota )_{\iota}$ such that 
$\sum_\iota x_\iota C(\iota ) \geq 0$, namely, we have

$$\inf_{x\, :\, \sum_\iota x_\iota C(\iota ) \geq 0}\sum_\iota b_\iota x_\iota \geq 0.$$

\emph{(b)} There exist solutions $X> 0$ of the system of equations (\ref{transl}) if
and only if

\noindent $\sum_\iota b_\iota x_\iota  > 0$ for all $(x_\iota )_{\iota}\not =0$ such that $\sum_\iota x_\iota C(\iota ) \geq 0$, namely for any
norm  $\| \, \cdot\, \|$ 

$$\inf_{x\, :\, \sum_\iota x_\iota C(\iota ) \geq 0,\, \| x\| =1}\sum_\iota b_\iota x_\iota >0.$$
\end{proposition}

\begin{proof}
(a) Assume that $\inf_{x\, :\, \sum_\iota x_\iota C(\iota ) \geq 0}\sum_\iota b_\iota x_\iota \geq 0$. The intersection of the closed convex cone of all positive semidefinite $p\times p$ matrices and the linear span $S$ of the $C(\iota )$'s contains a point that is interior to the cone, namely a positive matrix. The linear functional $l\, :\, \sum_\iota x_\iota C(\iota ) \mapsto \sum_\iota b_\iota x_\iota$ is well defined, and $ \geq 0$ on this intersection. By Mazur's theorem, see for instance \cite{AliprantisTourky}, \cite{Kantorovich}, it has a linear extension $L$ to the space $M_{p}^s$ of all selfadjoint matrices in $M_p$, such that $LY\geq 0$ for all $Y\geq 0$ in $M_{p}^s$. Now $L$ has the form $LY=\tr\, (XY) $ for some $X\in M_{p}^s$.  Letting $Y=\langle \,\cdot \, ,h\rangle h$ for an arbitrary vector $h\in \mathbb{C}^p$ gives $\langle Xh,h\rangle \geq 0$. Hence $X\geq 0$. Since  $L|_S =l$, for every $\iota $ we have $C(\iota ) \in S$ and $\tr\, (C(\iota ) X)=LC(\iota ) =l C(\iota ) =b_\iota $.
Conversely, suppose that  there exists an $X\geq 0$ such that $\tr\, (C(\iota ) X)=b_\iota $ for all $\iota$. Then for every $(x_\iota )_\iota$ such that $\sum_\iota x_\iota C(\iota ) \geq 0$, we have $\sum_\iota b_\iota x_\iota =\sum_\iota \tr\, (C(\iota ) X)x_\iota =\tr\, (X\sum_\iota x_\iota C(\iota ) )=\tr\, (X^{1/2}\sum_\iota x_\iota C(\iota ) X^{1/2})\geq 0$.

(b) Assume that $\inf_{x\, :\, \sum_\iota x_\iota C(\iota ) \geq 0, \, \| x\| =1}\sum_\iota 
b_\iota x_\iota > 0$. We proceed as above, except we need the following fact: given 
a finite dimensional real space $M$, a linear subspace $S$ and a closed convex 
cone $C\subset M$ such that $C\cap (-C)=\{ 0\}$,   any linear functional $l$ on $S$ 
such that $ls>0$ for all $s\not =0$ from $S\cap C$ has a linear extension $L$ to $M$ 
such that $Lm>0$ for all $m\not =0$ from $C$.
This is rather a known consequence of the Hahn-Banach, Mazur - type theorems, 
see for instance \cite{rh}.  
The necessity of the condition follows as in the case (a). 
\end{proof}

\subsection{The case of commutative data}\label{ss:ccd}
As mentioned before, Problem A was raised in  \cite{LiPoon} where the  commutative 
case
was proven to be  equivalent   to a linear programming problem, concerned with 
solving systems in nonnegative  variables. Our present approach  allows us to put 
their result into a new perspective.

Firstly, by the commutativity assumption we can suppose, without loss of generality,  
that all matrices $A_\nu ,B_\nu$ are diagonal.
For any matrix $u=[u_{ij}]_{i,j}$, set $\tilde{u}=[u_{ij}\delta_{ij}]_{i,j}$. 
The Proposition from below shows that in the equations
(\ref{tens}) we can replace then any positive semidefinite solution
$X $  by the (positive semidefinite) diagonal matrix 
$\tilde{X}=\mathrm{diag} (x_1 ,\ldots ,x_q )$, which leads to the simpler problem of finding 
some numbers $x_i \geq 0$ satisfying a linear system of 
nonhomogeneous  equations.

\begin{proposition}\label{p:c}
Let the matrices $A_\nu$, $B_\nu$  be diagonal. If $X\geq 0$ satisfies  (\ref{tens}), that is,
$\tr\, [(A_\nu \otimes   E_{ml}^{(k)}){X} ]=B_{\nu , lm}$, then 
$\widetilde{X}$ also is a solution to these equations, namely we have $\tr\, [(A_\nu \otimes   E_{ml}^{(k)})\widetilde{X} ]=B_{\nu , lm}$. 
\end{proposition}

\begin{proof}  Represent $X\in M_{nk}\equiv M_n \otimes M_k$ as $X=\sum_\mu Y_
\mu \otimes Z_\mu $ with $Y_\mu \in M_n$ and $Z_\mu \in M_k$. Using the easily 
checked formula $\widetilde{u\otimes v}=\tilde{u}\otimes \tilde{v}$,
we obtain $\tilde{X}=\sum_\mu \tilde{Y}_\mu \otimes \tilde{Z}_\mu $.
Hence, the  equality in the conclusion holds for $l\not =m$
 by inserting $\widetilde{X}$ in the left hand side, then using the formula 
 $\tr (u\otimes v)=\tr( u) \, \tr(v)$ and the equalities  $\tr\, (E_{lm}^{(k)}\tilde{Z}_\mu )
 =0$, $B_{\nu , lm}=0$.
To prove it also for $l=m$, 
use again $\tr (u\otimes v)=\tr\, u \, \tr\, v$
to write the desired conclusion in the form  $\sum_\mu \tr\, 
(A_\nu \tilde{Y}_\mu )Z_{\mu ,ll}=B_{\nu , ll}$. 
This is equivalent, by means of the equalities
$\tilde{A}_\nu =A_\nu$ and  the formula $\tr\, (\tilde{u}\tilde{v})= \tr (\tilde{u}v)$, to 
 $\tr\, [(A_\nu \otimes E_{ll}^{(k)})\sum_\mu {Y}_\mu \otimes {Z}_\mu ]=B_{\nu , ll}$, that is the case $l=m$ of (\ref{tens}) and so holds true by hypotheses. \end{proof}


\begin{thebibliography}{99}

\bibitem{AliprantisTourky} \textsc{C.D. Aliprantis, R. Tourky}, \textit{Cones and 
Duality}, Graduate Studies in Mathematics, Vol. 84, Amer. Math. Soc., Prividence 
R.I. 2007.

\bibitem{Arveson} \textsc{W.B.~Arveson}, Subalgebras of $C^*$-algebras. I,  
\textit{Acta Math.}  \textbf{123}(1969), 141--224. 

\bibitem{ca} \textsc{C. Ambrozie}, Finding positive matrices subject to linear 
restrictions, {\it Linear Algebra Appl.} {\bf 426}(2007), 716--728.

\bibitem{rh} \textsc{C. Ambrozie},  
A Riesz-Haviland type result for truncated moment problems with solutions in $L^1$, 
{\it J. Operator Theory}, accepted.

\bibitem{prima}\textsc{C. Ambrozie, A. Gheondea}, An interpolation problem for 
completely positive maps on matrix algebras; solvability and parametrization,
arXiv:1308.0667 [math.OA]

\bibitem{Barvinok} \textsc{A. Barvinok}, {\it A Course in Convexity}, Graduate Studies 
in Mathematics vol. 54. AMS, Providence, Rhode Island 2002.

\bibitem{borch} \textsc{B. Borchers, J.G. Young},
Implementation of a primal-dual method for SDP on a shared memory parallel
architecture, \textit{Comput. Optim. Appl.} \textbf{37}(2007), 355--369.

\bibitem{BoydGhaoui} \textsc{S. Boyd, L.E. Ghaoui},  Method of centers for minimizing generalized eigenvalues, {\it Linear Algebra Appl.} \textbf{188}:9(1993), 63--111.

\bibitem{BdGhFrBl} \textsc{S. Boyd, L.E. Ghaoui, E. Feron, V. Balakrishnan},   {\it Linear matrix inequalities in system and control theory}, SIAM, Philadelphia, 1994.

\bibitem{HrLm} \textsc{J.B. Hiriart-Urruty, C. Lemarechal}, {\it Convex Analysis and Minimization Algorithms I}, Springer-Verlag, Berlin Heidelberg 1993.

\bibitem{Choi} \textsc{M.-D. Choi}, Completely positive linear maps on 
complex matrices,  \textit{Lin. Alg. Appl.}  \textbf{10}(1975), 
285--290.

\bibitem{DArianoLoPresti} \textsc{G.M. D'Ariano, P. Lo Presti}, 
Tomography of quantum operations, \textit{Phys. Rev. Lett.} \textbf{86}(2001), 
4195--4198.

\bibitem{GolubVanLoan} \textsc{G. Golub, C. Van Loan}, {\it Matrix Computations}. 
The John Hopkins Univ. Press, Baltimore and London, 1989.

\bibitem{Goncalves} \textsc{D.S. Gon\c calves, C. Lavor, M.A. Gomes-Ruggiero, 
A.T. Ces\'ario, R.O. Vianna, T.O. Maciel}, Quantum state tomography with 
incomplete data: maximum entropy and variational quantum tomography, 
\textit{Phys. Rev. A} \textbf{87}(2013), 052140.

\bibitem{hei} \textsc{T. Heinosaari, M.A. Jivulescu, D. Reeb, M.M. Wolf}, Extending 
quantum operations, {\it J. Math. Phys.} {\bf 53} (2012), 102--208.

\bibitem{Hill} \textsc{R.D. Hill}, Linear transformations which preserve Hermitian 
matrices, \textit{Linear Algebra Appl.} \textbf{6}(1973), 257--262.

\bibitem{HLPZ} \textsc{Z. Huang, C.-K. Li, E. Poon, N.-S. Sze}, 
Physical transformations between quantum states,
\textit{J. Math. Phys.} \textbf{53}(2012), 102209.

\bibitem{Jamiolkowski} \textsc{A. Jamio\l kowski}, Linear transformations which 
preserve trace and positive semidefiniteness of operators, \textit{Rep. Math.
Phys.} \textbf{3}(1972), 275--278.

\bibitem{Kantorovich} \textsc{L. Kantorovich},
On the moment problem for a finite interval (Russian),
\textit{Dokl. Acad. Sci. SSR} \textbf{14}(1937), 531--537.

\bibitem{KocvaraStingl} \textsc{M. Ko\v{c}vara,  M. Stingl},  PENNON: Software for 
linear and nonlinear matrix inequalities, in \textit{Handbook on semidefinite, conic 
and polynomial optimization}, pp. 755--791, Internat. Ser. Oper. Res. Management 
Sci. Vol. 166, Springer, Berlin 2012.

\bibitem{Kra71} \textsc{K.~Kraus},
General state changes in quantum theory,
\textit{Ann. Physics} \textbf{64}(1971), 311--335. 

\bibitem{LundquistJohnson} \textsc{M.E. Lundquist, C.R. Johnson}, 
Linearly constrained positive definite completions, {\it Linear Algebra Appl.} 
{\bf 150}(1991), 195--207.

\bibitem{NesterovNemirovsky} \textsc{Y. Nesterov, A. Nemirovsky}, \textit{Interior 
Point Polynomial Algorithms in Convex Programming}, vol. 13, Studies in Applied 
Mathematics, SIAM, Philadelphia, PA, 1994.

\bibitem{LiPoon} \textsc{C.-K.~Li, Y.-T.~Poon}, Interpolation problem by completely 
positive maps, \textit{Linear Multlinear Algebra}, \textbf{59}(2011), 1159--1170.

\bibitem{Pau} \textsc{V.~Paulsen}, 
\textit{Completely Bounded Maps and Operator Algebras}, 
Cambridge University Press, Cambridge 2002.

\bibitem{Pillis} \textsc{J. de Pillis},
Linear transformations which preserve hermitian and positive semidefinite operators,
\textsc{Pacific J. Math.} \textbf{23}(1967), 129--137. 

\bibitem{SadatiYousefi} \textsc{N. Sadati, M.I. Yousefi}, {\it A nonlinear SDP approach 
for matrix rank minimization problems with applications}, Industrial Electronics and 
Control Applications, ICIECA Quito, 2005.

\bibitem{SmithWard} \textsc{R.R. Smith, J.D. Ward}, Matrix ranges for Hilbert space
operators, \textit{Amer. J. Math.} \textbf{102}(1980), 1031--1081.

\bibitem{VanderbergheBoyd} \textsc{L. Vanderberghe, S, Boyd}, Semidefinite 
programming, \textit{SIAM Review}, \textbf{38}(1996), 49--95.

\end{thebibliography}
\end{document}